\documentclass{article}

\usepackage{arxiv}

\usepackage{graphicx}
\usepackage{amsfonts,amsmath,amstext,amsthm,amssymb}

\numberwithin{equation}{section}

\usepackage{amsfonts}
\usepackage{amsmath}
\usepackage{amscd}
\usepackage{amsthm}
\usepackage{amssymb}
\usepackage{mathrsfs}
\usepackage{pspicture}
\usepackage{graphics}
\usepackage{pstricks}
\usepackage{pst-plot}
\usepackage{upgreek}
\usepackage{dirtytalk}

\newtheorem{theorem}{Theorem}[section]
\newtheorem{lemma}{Lemma}[section]

\newtheorem{definition}{Definition}[section]

\newtheorem{OldTheorem}{Theorem}

\DeclareMathOperator{\sgn}{sgn}
\newcommand {\V }[1]{{\rm V}\left(#1\right)}
\def\BV{{\rm BV\,}}

\def\PI{\Uppi}
\def\TPI{\tilde\Uppi}
\def\FI{\varphi_*}

\def\ZC{\ensuremath{\mathbb C}}

\def\ZI{\ensuremath{\mathbb I}}

\def\ZN{\ensuremath{\mathbb N}}
\def\zP{\ensuremath{\mathcal P}}

\def\ZR{\ensuremath{\mathbb R}}

\def\ZT{\ensuremath{\mathbb T}}
\def\ZZ{\ensuremath{\mathbb Z}}
\def\deff{\overset{\text{def}}{=}}

\def\a{\textsc{\char13}}

\def\md#1#2\emd{
\ifx0#1\begin{equation*} #2 \end{equation*}\fi  
\ifx1#1\begin{equation}#2\end{equation}\fi   
\ifx2#1\begin{align*}#2\end{align*}\fi   
\ifx3#1\begin{align}#2\end{align}\fi    
\ifx4#1\begin{gather*}#2\end{gather*}\fi  
\ifx5#1\begin{gather}#2\end{gather}\fi   
\ifx6#1\begin{multline*}#2\end{multline*}\fi  
\ifx7#1\begin{multline}#2\end{multline}\fi  
\ifx8#1\begin{equation*}\begin{split}#2\end{split}\end{equation*}\fi
\ifx9#1\begin{equation}\begin{split}#2\end{split}\end{equation}\fi
}
\makeatletter
\def\env@cases{%
  \let\@ifnextchar\new@ifnextchar
  \left\lbrace
  \def\arraystretch{0.9}%
  \array{@{}l@{\quad}l@{}}}
\makeatother
\newcommand{\e }[1]{(\ref{#1})}
\newcommand{\lem }[1]{Lemma \ref{#1}}
\newcommand{\trm }[1]{Theorem \ref{#1}}
\newcommand{\otrm }[1]{Theorem \ref{#1}}

\title{On Generalizations of Fatou's Theorem in $L^p$ for Convolution Integrals with General Kernels}%

\author{
  Mher Safaryan \\
  Yerevan State University, Armenia\\
  Institute of Mathematics, Armenian Academy of Sciences\\
  KAUST, KSA \\
  \texttt{mher.safaryan@gmail.com}
}

\begin{document}
\maketitle

\begin{abstract}
We prove Fatou-type theorem on almost everywhere convergence of convolution integrals in spaces $L^p\,(1<p<\infty)$ for general kernels, forming an approximate identity. For a wide class of kernels we show that obtained convergence regions are optimal in some sense. It is established a weak boundedness in $L^p\,(1\le p<\infty)$ of the corresponding maximal operator.
\end{abstract}

\section{Introduction}\label{fatou-intro}
The following remarkable theorems of Fatou \cite{Fat} play significant role in the study
of boundary value problems of analytic and harmonic functions.

\begin{OldTheorem}[Fatou, 1906]\label{OldFatou-1}
Any bounded analytic function on the unit disc $D=\{z\in \mathbb{C}:\, |z|<1\}$ has non-tangential limit for almost all boundary points.
\end{OldTheorem}
\begin{OldTheorem}[Fatou, 1906]\label{OldFatou-2}
If a function $\mu$ of bounded variation is differentiable at $x_0\in \ZT$, then the Poisson integral
\md0
\zP_r(x,d\mu)=\frac{1}{2\pi }\int_\ZT \frac{1-r^2}{1-2r\cos (x-t)+r^2} d\mu(t)\
\emd
converges non-tangentially to $\mu'(x_0)$ as $r\to 1$.
\end{OldTheorem}

These two fundamental theorems, have many applications in different mathematical theories including analytic functions, Hardy spaces, harmonic analysis, differential equations and etc. There are various generalization of these theorems in different aspects. Almost everywhere convergence over some tangential approach regions investigated by Nagel and Stein \cite{NaSt}, Di Biase \cite{DiBi, DiBi2}, Di Biase-Stokolos-Svensson-Weiss \cite{BSSW}.
Sj\"{o}gren \cite{Sog1,Sog2,Sog3}, R\"{o}nning \cite{Ron1,Ron2,Ron3}, Katkovskaya-Krotov \cite{Kro}, Krotov \cite{Kro2}, Brundin \cite{Bru}, Mizuta-Shimomura \cite{MizShi}, Aikawa \cite{Aik3} studied fractional Poisson integrals with respect to the fractional power of the Poisson kernel and obtained some tangential convergence properties for such integrals. More precisely they considered the integrals

\md0
\zP_r^{(1/2)}(x,f) \deff \int_\ZT P^{(1/2)}_r(x-t) f(t)\,dt = \frac{1}{c(r)}\int_\ZT [P_r(x-t)]^{1/2} f(t)\,dt,
\emd
where
\md0
P_r(x) = \frac{1-r^2}{1-2r\cos x+r^2}, \quad 0<r<1, \quad x\in\ZT
\emd
is the Poisson kernel for the unit disk and
\md0
c(r)=\int_\ZT [P_r(t)]^{1/2}\,dt \asymp (1-r)^{1/2}\log\frac{1}{1-r}
\emd
is the normalizing coefficient. Here, the notation $A\asymp B$ means double inequality $c_1 A \le B\le c_2 A$ for some positive absolute constants $c_1$ and $c_2$, which might differ in each case.

\begin{OldTheorem}[see \cite{Sog1,Ron1,Ron2}]\label{OldRonningSogren}
For any $f\in L^p(\ZT),\, 1\le p\le \infty$
\md1\label{0.frac-poisson}
\lim_{r\to 1}\zP_r^{(1/2)} (x+\theta(r),f) = f(x)
\emd
almost everywhere $x\in\ZT$, whenever
\md1\label{0.theta-bound-P0.5}
|\theta(r)|\le
\begin{cases}
c(1-r)\left(\log\frac{1}{1-r}\right)^p & \quad\text{if}\quad 1\le p<\infty,\\
c_\alpha(1-r)^\alpha,\,\text{for any }\,0<\alpha<1 & \quad\text{if}\quad p=\infty,
\end{cases}
\emd
where $c_\alpha>0$ is a constant, depended only on $\alpha$.
\end{OldTheorem}
The case of $p=1$ is proved in \cite{Sog1}, $1<p\le \infty$ is considered in \cite{Ron1}, \cite{Ron2}. Moreover, in \cite{Ron1} weak type inequalities for the maximal operator of square root Poisson integrals are established.
\begin{OldTheorem}[R\"{o}nning, 1997]\label{OldRonning}
Let $1<p<\infty$. Then the maximal operator
\md0
\zP^*_{1/2}(x,f) = \sup_{\substack{|\theta|<c(1-r)\left(\log\frac{1}{1-r}\right)^p\\1/2<r<1}} \zP_r^{(1/2)}(x+\theta,|f|)
\emd
is of weak type $(p,p)$.
\end{OldTheorem}
In \cite{Kro} weighted strong type inequalities for the same operators are established. Related questions were considered also in higher dimensions. Saeki \cite{Sae} studied Fatou type theorems for non-radial kernels. Kor\'anyi \cite{Kor} extended Fatou\a s theorem for the Poisson-Szeg\"{o} integral. In \cite{NaSt} Nagel and Stein proved that the Poisson integral on the upper half space of $\ZR^{n+1}$ has the boundary limit at almost every point within a certain approach region, which is not contained in any non-tangential approach regions. Sueiro \cite{Sue} extended Nagel-Stein\a s result for the Poisson-Szeg\"{o} integral. Almost everywhere convergence over tangential tress (family of curves) were investigated by Di Biase \cite{DiBi}, Di Biase-Stokolos-Svensson-Weiss \cite{BSSW}. In \cite{Kro} and \cite{Aik3} higher dimensional cases of fractional Poisson integrals are studied as well.

The current paper is the development of the authors investigation in \cite{KarSaf1}. In \cite{KarSaf1} we introduced $\lambda(r)-$convergence, which is a generalization of non-tangential convergence in the unit disc, where $\lambda(r)$ is a function
\md1\label{0.lambdar-def}
\lambda:(0,1)\to(0,\infty) \quad\text{with}\quad \lim_{r\to1}\lambda(r)=0.
\emd
Let $\ZT = \ZR/2\pi \ZZ$ be the unit circle. For a given $x\in \ZT$ we define $\lambda(r,x)$ to be the interval $[x-\lambda(r), x+\lambda(r)]$ in $\ZT$. In case of $\lambda(r)\ge \pi$ we assume $\lambda(r,x)=\ZT$. Let $F_r(x)$ be a family of functions from $L^1(\ZT)$, where $r$ varies in $(0,1)$. We say $F_r(x)$  is $\lambda(r)-$convergent at a point $x\in \ZT$ to a value $A$, if
\md0
\lim_{r\to 1}\sup_{\theta\in \lambda(r,x)}|F_r(\theta)-A|=0.
\emd
Otherwise this relation will be denoted by
\md1\label{1.lambdar-def-1}
\lim_{\stackrel{r\to 1}{\theta\in\lambda(r,x)}}F_r(\theta)=A.
\emd
We say $F_r(x)$ is $\lambda(r)-$divergent at $x\in\ZT$ if \e{1.lambdar-def-1} does not hold for any $A\in\ZR$.

There are at least two ways to interpret $\lambda(r)-$convergence. First, we can associate the function $\lambda(r)$ with regions
\md0
\Omega^x_\lambda \deff \{r e^{i\theta}\in\ZC\colon r\in(0,1),\,|\theta-x|<\lambda(r)\}\subset D,\quad x\in\ZT.
\emd
Then $\lambda(r)-$convergence for $F_r(x)$ at some point $x\in\ZT$ becomes convergence over the region $\Omega^x_\lambda$ for $\tilde{F}(r e^{i x})=F_r(x)$. It is clear, that the non-tangential convergence in the unit disc is the case of $\lambda(r)=c(1-r)$. Second, we can think of it as one dimensional \say{pointwise-uniform} convergence on $\ZT$, meaning that $\lambda(r)-$convergence at a point $x\in\ZT$ depends only on values of functions on $\lambda(r,x)$ which contracts to $x$.

Denote by $\BV(\ZT)$ the functions of bounded variation on $\ZT$. Any given function of bounded variation $\mu\in\BV(\ZT)$ defines a Borel measure on $\ZT$. We consider the family of integrals
\md1\label{1.Phi}
\Phi_r(x,d\mu) \deff \int_\ZT \varphi_r(x-t)\,d\mu(t), \quad \mu\in\BV(\ZT),
\emd
where $0<r<1$ and kernels $\varphi_r\in L^\infty (\ZT)$ form an \emph{approximate identity} defined as follows:

\begin{definition}
We define an approximate identity as a family $\left\{\varphi_r\right\}_{0<r<1}\subset L^{\infty}(\ZT)$ of functions satisfying the following conditions:
\begin{itemize}
\item[$\Phi1.$] $\int_\ZT \varphi_r(t)\,dt\to 1$  as $r\to 1$,
\item[$\Phi2.$] $\varphi_r^*(x) \deff \sup\limits_{|x|\le |t|\le \pi}|\varphi_r(t)|\to 0$ as $r\to 1,\quad 0<|x|\le\pi,$
\item[$\Phi3.$] $C_\varphi \deff \sup\limits_{0<r<1}\|\varphi_r^*\|_1<\infty.$
\end{itemize}
\end{definition}

Approximate identities with the above definition were investigated in \cite{Ben, Katz, KarSaf1}.
Notation $\varphi_r$ should not be confused with the classical dilation approximate identities \cite{Ste}.
In case of $\mu$ is absolutely continuous and $d\mu(t)=f(t)dt$ for some $f\in L^p(\ZT),\,1\le p\le\infty$, then the integral \e{1.Phi} will be denoted as $\Phi_r(x,f)$.

Carlsson \cite{Car} obtained almost everywhere convergence result for non-negative approximate identities with regular level sets, which is defined by the following condition:
\begin{equation*}
\sup\left\{|x| \colon x\in L(r,s)\right\} \le C|L(r,s)|, \text{ for all } 0<r<1,\,s>0,
\end{equation*}
where $C$ is some constant and $L(r,s) = \left\{x\in\ZT \colon \varphi_r(x)>s \right\}$.

\begin{OldTheorem}[Carlsson, 2008]\label{OldCarlsson}
Let $\{\varphi_r(x)\ge 0\}$ be a non-negative approximate identity with regular level sets and $\rho(r) = \|\varphi_r\|_q^{-p}$, where $1\le p<\infty$ and $q=p/(p-1)$ is the conjugate index of $p$. Then for any $f\in L^p(\ZT)$
\md0
\lim_{\substack{r\to1\\|\theta|<c \rho(r)}} \Phi_r(x+\theta,f) = f(x),
\emd
almost everywhere $x\in\ZT$.
\end{OldTheorem}

Although \trm{OldCarlsson} gives a general connection between approximate identities and convergence regions, we will see that it can be extended to any approximate identity without regular level sets assumption. Moreover, obtained convergence regions are shown to be optimal for a wide class of kernels.
Here, the optimality of convergence regions is considered within the regions $\Omega_\lambda^x$ with $x\in\ZT$ and $\lambda(r)$ satisfying (\ref{0.lambdar-def}). More precisely, the optimality of convergence regions is understood as the optimality of the rate of $\lambda(r)$ (when $r\to1$) ensuring almost everywhere $\lambda(r)-$convergence.

In \cite{KarSaf1} we proved that the condition $\PI(\lambda,\varphi) < \infty$ with
\md0
\PI(\lambda,\varphi) = \limsup_{r\to 1}\lambda(r) \|\varphi_r\|_\infty
\emd
is necessary and sufficient for almost everywhere $\lambda(r)-$convergence of the integrals $\Phi_r(x,d\mu)$, $\mu\in\BV(\ZT)$ as well as $\Phi_r(x,f),\,f\in L^1(\ZT)$. Moreover, we proved that convergence holds at any point where $\mu$ is differentiable for the integrals $\Phi_r(x,d\mu)$ and at any Lebesgue point of $f\in L^1(\ZT)$ for the integrals $\Phi_r(x,f)$.
Thus, the condition $\PI(\lambda, \varphi)<\infty$ determines the exact rate of $\lambda(r)$ function, ensuring such convergence. In this case the rate depends only on $\|\varphi_r\|_\infty$. If the kernel $\varphi_r$ coincides with the Poisson kernel $P_r$, then $\|P_r\|_\infty \asymp \frac{1}{1-r}$ and the bound $\PI(\lambda,P)<\infty$ coincides with the well-known condition
\md1\label{0.non-tangential-condition}
\limsup_{r\to1}\frac{\lambda(r)}{1-r} < \infty,
\emd
guaranteeing non-tangential convergence in the unit disk. Furthermore, if we take the fractional Poisson kernel $P^{(1/2)}_r$, then
\md1\label{sqrootP-infnorm}
\|P^{(1/2)}_r\|_\infty = \frac{1}{c(r)}\|P^{1/2}_r\|_\infty \asymp \left((1-r)\log\frac{1}{1-r}\right)^{-1}
\emd
and we deduce \e{0.frac-poisson} when $p=1$ with an additional information about the points where the convergence occurs.

In the same paper \cite{KarSaf1}, an analogous necessary and sufficient condition was established also for almost everywhere $\lambda(r)-$convergence of $\Phi_r(x,f),\,f\in L^\infty(\ZT)$ with condition $\PI_\infty(\lambda,\varphi)=0$, where
\md0
\PI_\infty(\lambda,\varphi) = \limsup_{\delta \to 0}\limsup_{r\to 1} \int_{-\delta \lambda(r)}^{\delta \lambda(r)}\varphi_r(t)dt.
\emd
In addition, we proved that convergence holds at any Lebesgue point of $f\in L^\infty(\ZT)$.

One can easily check that in the case of Poisson kernel $P_r(t)$, for a given function $\lambda(r)$ with \e{0.lambdar-def}, the value of $\PI_\infty(\lambda,P)$ can be either $0$ or $1$, where the condition $\PI_\infty(\lambda,P)=0$ is equivalent to \e{0.non-tangential-condition}, and $\PI_\infty(\lambda,P)=1$ coincides with
\md0
\limsup_{r\to1}\frac{\lambda(r)}{1-r} = \infty.
\emd
If $\lambda(r)$ satisfies the condition \e{0.theta-bound-P0.5} with $p=\infty$, then simple calculations show that for such $\lambda(r)$ and for the square root Poisson kernel $P^{(1/2)}_r(t)$ we have $\PI_\infty(\lambda,P^{(1/2)})=0$. Hence we deduce \e{0.frac-poisson} when $p=\infty$ with an additional information about the points where the convergence occurs. Taking $\lambda(r)=(1-r)^\alpha$ with a fixed $0<\alpha<1$ we will get $\PI_\infty(\lambda,P^{(1/2)})=1-\alpha>0$, which implies the optimality of the bound \e{0.theta-bound-P0.5} in the case $p=\infty$ too.

\section{Main Results}
In this paper, we obtain similar results for the integrals $\Phi_r(x,f),\,f\in L^p(\ZT),\,1<p<\infty$ with condition $\PI_p(\lambda,\varphi)<\infty$, where
\md0
\PI_p(\lambda,\varphi) = \limsup_{r\to 1}\lambda(r) \|\varphi_r\|_\infty \FI^{p-1}(r), \quad \FI(r) = \sup_{x\in\ZT}|x \varphi^*_r(x)|.
\emd
\begin{theorem}\label{1.fatouLp}
Let $\{\varphi_r\}$  be an arbitrary approximate identity and $\lambda(r)$ satisfies the condition $\PI_p(\lambda, \varphi) < \infty$ for some $1<p<\infty$. Then for any $f\in L^p(\ZT)$
\md0
\lim_{\stackrel{r\to 1}{y\in\lambda(r,x)}} \Phi_r(y,t) = f(x),
\emd
almost everywhere $x\in\ZT$.
\end{theorem}

The proof of this theorem is established by standard methods using weak type inequality of the associated maximal operator $\Phi_{\lambda}^*$ defined as
\md0
\Phi_{\lambda}^*(x,f) = \sup_{\substack{|x-y|<\lambda(r)\\0<r<1}}|\Phi_r(y,f)|
= \sup_{\substack{|x-y|<\lambda(r)\\0<r<1}}\left|\int_\ZT \varphi_r(y-t)f(t)\,dt \right|.
\emd

\begin{definition}
For $f\in L^1(\ZT)$ denote by $Mf$ the Hardy-Littlewood maximal function defined as follows:
\begin{equation*}
Mf(x) = \sup_{0<t<\pi} \frac{1}{2t}\int_{x-t}^{x+t} |f(t)|\,dt, \quad x\in\ZT.
\end{equation*}
\end{definition}

\begin{theorem}\label{1.weakpp}
Let $\{\varphi_r\}$ be an arbitrary approximate identity and for some $1\le p < \infty$ the function $\lambda(r)$ satisfies $\TPI_p(\lambda, \varphi)<\infty$, where
\md0
\TPI_p(\lambda, \varphi) = \sup_{0<r<1} \lambda(r)\|\varphi_r\|_{\infty}\FI(r)^{p-1}.
\emd
Then for any $f\in L^p(\ZT)$
\md1\label{1.Phi*-bound}
\Phi_{\lambda}^*(x,f) \le C \left(M|f|^p(x)\right)^{1/p},\quad x\in\ZT,
\emd
where the constant $C$ does not depend on function $f$. In particular, the operator $\Phi_{\lambda}^*$ is of weak type $(p,p)$, i.e.
\md0
\left|\left\{x\in\ZT \colon \Phi_{\lambda}^*(x,f) > t\right\}\right| \le \frac{\tilde{C}}{t^p}\|f\|_p^p
\emd
holds for any $t>0$, where constant $\tilde{C}$ does not depend on function $f$ and $t$.
\end{theorem}

The following theorem reveals significance of the condition $\PI_p(\lambda, \varphi) < \infty$ in \trm{1.fatouLp} with an additional constraint on kernels.

\begin{theorem}\label{1.fatouLp-example}
Let $\{\varphi_r\}$  be an arbitrary approximate identity with $c_\varphi>0$, where
\md1\label{1.extraCondition-phi}
c_\varphi = \liminf_{r\to1}\frac{1}{\FI(r)}\int_{-\mu(r)}^{\mu(r)}|\varphi_r(t)|\,dt, \quad \mu(r)=\frac{\FI(r)}{\|\varphi_r\|_\infty},
\emd
and $\lambda(r)$ satisfies the condition $\PI_p(\lambda, \varphi)=\infty$ for some $p,\,1<p<\infty$. Then there exists a function $f\in L^p(\ZT)$ such that
\md1\label{1.Lp-divergence}
\limsup_{\stackrel{r\to 1}{y\in\lambda(r,x)}} \Phi_r(y,f) = \infty,
\emd
for all $x\in\ZT$.
\end{theorem}

As we will see in \lem{1.phi-r-bound}, the function $\FI(r)$ satisfies
\md1\label{0.phi-r-bound-ineqs}
\frac{c}{\log\|\varphi_r\|_\infty} \le \FI(r) \le C_\varphi, \quad r_0<r<1,
\emd
where $c$ is a positive absolute constant. First of all, both bounds in \e{0.phi-r-bound-ineqs} are accessible. For instance, if we take the Poisson kernel $P_r(t)$ then it can be checked that $P_*(r)\asymp1$. On the other hand, if we take the square root Poisson kernel $P^{(1/2)}_r(t)$, then one can show that
\md1\label{0.square-root-poisson}
P^{(1/2)}_*(r) \asymp \left(\log\frac{1}{1-r}\right)^{-1} \asymp \frac{1}{\log\|P^{(1/2)}_r\|_\infty}.
\emd

From the first inequality of \e{0.phi-r-bound-ineqs} it follows that the condition $\PI_p(\lambda,\varphi)<\infty$ cannot be weaker (in other words the associated region of convergence in the unit disk cannot be larger) than
\md0
\limsup_{r\to 1}\lambda(r) \|\varphi_r\|_\infty \left(\frac{1}{\log\|\varphi_r\|_\infty}\right)^{p-1}<\infty,
\emd
which again depends only on values $\|\varphi_r\|_\infty$. The second inequality of \e{0.phi-r-bound-ineqs} ensures that the multiplier $\FI(r)$ in condition $\PI_p(\lambda,\varphi)<\infty$ can only weaken that condition (in other words can only enlarge the associated region of convergence in the unit disk) if we increase $p$, i.e. condition $\PI_{p_1}<\infty$ imples $\PI_{p_2}<\infty$, whenever $1\le p_1 \le p_2 < \infty$.

Taking into account \e{sqrootP-infnorm} and \e{0.square-root-poisson}, note that these results imply \e{0.frac-poisson} when $1<p<\infty$ as well as \otrm{OldRonning}. Besides, for the Poisson kernel or for the square root Poisson kernel we have $c_\varphi=1>0$, and from \trm{1.fatouLp-example} we conclude the optimality of the bound \e{0.theta-bound-P0.5} for $1<p<\infty$.

\section{Auxiliary Lemmas}\label{fatou-lemmas}
We will use the following lemma in the proof of \trm{1.fatouLp-example}.
\begin{lemma}\label{1.BV-means}
Let $\varphi\in\BV(\ZT)$ be a function of bounded variation and
\md0
\Delta_k = \bigcup_{j=0}^{n_k-1}\left[\frac{2\pi j}{n_k}-\delta_k,\frac{2\pi j}{n_k}+\delta_k\right] \subset\ZT,
\emd
where $n_k\in\ZN, \delta_k\in\ZT$ such that $n_k\to\infty$ as $k\to\infty$ and $0<\delta_k<\frac{\pi}{n_k},\,k=1,2,\dots$. Then
\md0
\lim_{k\to\infty} \frac{1}{|\Delta_k|}\int_{\Delta_k} \varphi(\theta+t)\,dt = \frac{1}{2\pi}\int_\ZT \varphi(t)\,dt,
\emd
where the convergence is uniform with respect to $\theta\in\ZT$.
\end{lemma}
\begin{proof}
Denote by $\Delta_k^j$ the $j$th component interval of $\Delta_k$ such that $\Delta_k=\cup_{0\le j<n_k}\Delta_k^j$. Condition $0<\delta_k<\frac{\pi}{n_k}$ implies that component intervals are pairwise disjoint and $|\Delta_k^j|=2\delta_k$. Let $\theta+\Delta_k^j = \{\theta+t\colon t\in\Delta_k^j\}$ and $\V{\varphi, [a,b]}$ be the total variation of function $\varphi$ on an interval $[a,b]\subset\ZT$. Then
\md9
&\left|\frac{1}{|\Delta_k|}\int_{\Delta_k} \varphi(\theta+t)\,dt - \frac{1}{2\pi}\int_\ZT \varphi(t)\,dt\right|\\
&\mspace{100mu}\le\left|\frac{1}{n_k}\sum_{j=0}^{n_k-1}\frac{1}{2\delta_k}\int_{\Delta_k^j}\varphi(\theta+t)\,dt
- \frac{1}{n_k}\sum_{j=0}^{n_k-1}\varphi\left(\theta+\frac{2\pi j}{n_k}\right)\right|\\
&\mspace{100mu} + \left|\frac{1}{n_k}\sum_{j=0}^{n_k-1}\varphi\left(\theta+\frac{2\pi j}{n_k}\right)
         - \frac{1}{2\pi}\int_\ZT \varphi(\theta+t)\,dt\right|\\
&\mspace{100mu}\le \frac{1}{n_k}\sum_{j=0}^{n_k-1}\frac{1}{2\delta_k}\int_{\Delta_k^j}\left|\varphi(\theta+t)-\varphi\left(\theta+\frac{2\pi j}{n_k}\right)\right|\,dt\\
&\mspace{100mu}+ \frac{1}{2\pi}\sum_{j=0}^{n_k-1}\int_{2\pi j/n_k}^{2\pi(j+1)/n_k}\left|\varphi(\theta+t)-\varphi\left(\theta+\frac{2\pi j}{n_k}\right)\right|\,dt\\
&\mspace{100mu}\le \frac{1}{n_k}\sum_{j=0}^{n_k-1}\V{\varphi, \theta+\Delta_k^j}
+ \frac{1}{n_k}\sum_{j=0}^{n_k-1} \V{\varphi, \left[\theta+\frac{2\pi j}{n_k}, \theta+\frac{2\pi(j+1)}{n_k}\right]}\\
&\mspace{100mu}\le \frac{2}{n_k}\V{\varphi, \ZT}.
\emd
The last term does not depend on $\theta$ and vanishes as $k\to\infty$, which completes the proof of the lemma.
\end{proof}

The next 3 lemmas are key ingredients of the proof of \trm{1.weakpp}.
\begin{lemma}\label{1.Phi*ineq1}
Let $\{\varphi_r\}$ be an arbitrary approximate identity and $\lambda(r)>0$ be any function. Then for any function $f\in L^1(\ZT)$
\md0
\sup_{\substack{|x-y|<\lambda(r)\\0<r<1}}\left|\int_{\lambda(r)\le|t|\le\pi} \varphi_r(t)f(y-t)\,dt \right| \le 8 C_\varphi\cdot Mf(x),\quad x\in\ZT.
\emd
\end{lemma}
\begin{proof}
Without loss of generality we may assume that $f$ is non-negative. Let $x, y\in\ZT,\,0<r<1$ such that $|x-y|<\lambda(r)$. We devide the interval $[\lambda(r),\pi]$ into $[2^{k-1}\lambda(r), 2^k\lambda(r)], k=1,2,\dots,Q=\lceil\log\frac{\pi}{\lambda(r)}\rceil$ and estimate the values of $\varphi_r(t)$ by its maximum in each divided interval:
\md2
\left|\int_{\lambda(r)}^{\pi} \varphi_r(t)f(y-t)\,dt\right|
&\le \sum_{k=1}^Q \int_{2^{k-1}\lambda(r)}^{2^k\lambda(r)} \varphi^*_r(t)f(y-t)\,dt\\
&\le \sum_{k=1}^Q \varphi^*_r\left(2^{k-1}\lambda(r)\right)\int_{2^{k-1}\lambda(r)}^{2^k\lambda(r)} f(y-t)\,dt\\
&\le \sum_{k=1}^Q \varphi^*_r\left(2^{k-1}\lambda(r)\right)\int_{\lambda(r)}^{2^k\lambda(r)} f(y-t)\,dt
\emd
Since $|x-y|<\lambda(r)$ we have
\md0
\int_{\lambda(r)}^{2^k\lambda(r)} f(y-t)\,dt \le \int_{0}^{(1+2^k)\lambda(r)} f(x-t)\,dt.
\emd
Therefore
\md2
\left|\int_{\lambda(r)}^{\pi} \varphi_r(t)f(y-t)\,dt\right|
&\le \sum_{k=1}^Q \varphi^*_r\left(2^{k-1}\lambda(r)\right)\int_{0}^{(1+2^k)\lambda(r)} f(x-t)\,dt\\
&\le Mf(x)\cdot \sum_{k=1}^Q \varphi^*_r\left(2^{k-1}\lambda(r)\right)(1+2^k)\lambda(r)\\
&\le 8 Mf(x)\cdot \sum_{k=0}^{Q-1} \varphi^*_r\left(2^k\lambda(r)\right)2^{k-1}\lambda(r)\\
&\le 8 Mf(x) \cdot \int_0^\pi \varphi^*_r(t)\,dt,
\emd
where in the last inequality we have used the following simple geometric inequlaity:
\md8
\varphi^*_r\left(\lambda(r)\right)\lambda(r) &+ \sum_{k=1}^{Q-1} \varphi^*_r\left(2^k\lambda(r)\right) 2^{k-1}\lambda(r)\\
&\le \int_0^{\lambda(r)}\varphi^*_r(t)\,dt + \sum_{k=1}^{Q-1}\int_{2^{k-1}\lambda(r)}^{2^k\lambda(r)}\varphi^*_r(t)\,dt \\
&\le \int_0^\pi \varphi^*_r(t)\,dt.
\emd
Thus we have
\md0
\left|\int_{\lambda(r)}^{\pi} \varphi_r(t)f(y-t)\,dt\right| \le 8 Mf(x)\cdot\int_0^\pi \varphi^*_r(t)\,dt.
\emd
In the same way we get
\md0
\left|\int_{-\pi}^{-\lambda(r)} \varphi_r(t)f(y-t)\,dt\right| \le 8 Mf(x)\cdot\int_{-\pi}^0 \varphi^*_r(t)\,dt.
\emd
Therefore
\md0
\sup_{\substack{|x-y|<\lambda(r)\\0<r<1}} \left|\int_{\lambda(r)\le|t|\le\pi} \varphi_r(t)f(y-t)\,dt\right|
\le 8 Mf(x)\cdot\sup_{0<r<1}\|\varphi^*_r\|_1 \le 8C_\varphi\cdot Mf(x).
\emd
\end{proof}

\begin{lemma}\label{1.lemma-TAbound}
Let $\{\varphi_r\}$ be an arbitrary approximate identity and $\mu(r),\,\lambda(r)$ are some functions with
\begin{itemize}
\item[\textit{1.}] $0<\mu(r)\le\lambda(r)\le\pi$,
\item[\textit{2.}] $\lambda(r) \le C\mu(r)\FI^{-p}(r)$, for some $C>0$ and $p\ge 1$.
\end{itemize}
Then for any $A\ge1$ and for any function $f\in L^p(\ZT)$
\md0
T_A f(x) \le \left(C\cdot\frac{M|f|^p(x)}{A}\right)^{1/p},\quad x\in\ZT,
\emd
where
\md4
T_A f(x) = \sup_{\substack{A\mu(r)<|x-y|<\lambda(r)\\0<r<1}}\FI(r)m_f(y,A\mu(r)),\\
m_f(y,t) = \frac{1}{2t}\int_{y-t}^{y+t}|f(u)|\,du.\nonumber
\emd
\end{lemma}
\begin{proof}
Without loss of generality we may assume that $f$ is non-negative. Using the definition of $T_A$ and Jensen\a s inequality we get
\md2
T_A^p f(x)
&= \sup_{\substack{A\mu(r)<|x-y|<\lambda(r)\\0<r<1}}\FI^p(r) m_f^p(y, A\mu(r))\\
&\le \sup_{\substack{A\mu(r)<|x-y|<\lambda(r)\\0<r<1}}\FI^p(r) m_{f^p}(y, A\mu(r))\\
&= \sup_{k\in\ZN} \sup_{\substack{2^{k-1}A\mu(r)<|x-y|\le 2^k A\mu(r)\\2^k A\mu(r)\le\lambda(r)\\0<r<1}} \FI^p(r) m_{f^p}(y, A\mu(r))
\emd
To estimate the inner supremum, first note that $2^k A\mu(r)\le\lambda(r)\le C\mu(r)\FI(r)^{-p}$ imples $\FI^p(r)\le C\left(2^k A\right)^{-1}$, where $C$ is the constant from condition $\mathit{2}$. Furthermore, since $2^{k-1}A\mu(r)<|x-y|\le2^k A\mu(r)$ we have
\md2
m_{f^p}(y,A\mu(r))
&= \frac{1}{2A\mu(r)} \int_{y-A\mu(r)}^{y+A\mu(r)}f^p(u)\,du\\
&\le \frac{1}{2A\mu(r)} \int_{x}^{x+(1+2^k)A\mu(r)}f^p(u)\,du\\
&\le \frac{(1+2^k)A\mu(r)}{2A\mu(r)} Mf^p(x) \le 2^k Mf^p(x).
\emd
Therefore
\md0
T_A^p f(x) \le \sup_{k\in\ZN} C\left(2^k A\right)^{-1}2^k Mf^p(x) = C\cdot\frac{Mf^p(x)}{A}.
\emd
\end{proof}

\begin{lemma}\label{1.lemma-supmulambda}
Let $\{\varphi_r\}$ be an arbitrary approximate identity and $\mu(r), \lambda(r)$ are some functions satisfying the conditions \textit{1.} and \textit{2.} from \lem{1.lemma-TAbound}. Then for any function $f\in L^p(\ZT)$
\md1\label{1.sup-mulambda}
\sup_{\substack{|x-y|<\lambda(r)\\0<r<1}}\left|\int_{\mu(r)\le|t|\le\lambda(r)} \varphi_r(t)f(y-t)\,dt \right|
\le \frac{4C^{1/p}}{2^{1/p}-1} \left(M|f|^p(x)\right)^{1/p},\quad x\in\ZT.
\emd
\end{lemma}
\begin{proof}
Again, we may assume that $f$ is non-negative. Let $x, y\in\ZT, 0<r<1$ and $|x-y|<\lambda(r)$. If $Q=\lceil\log\frac{\lambda(r)}{\mu(r)}\rceil$, we split the integral in \e{1.sup-mulambda} as follows
\md9\label{1.varphir-bound}
&\left|\int_{\mu(r)\le|t|\le\lambda(r)} \varphi_r(t)f(y-t)\,dt\right|\\
&\le \sum_{k=1}^Q \int_{2^{k-1}\mu(r)\le|t|\le2^k\mu(r)} \varphi^*_r(t)f(y-t)\,dt\\
&\le \sum_{k=1}^Q \max\left(\varphi^*_r\left(2^{k-1}\mu(r)\right), \varphi^*_r\left(-2^{k-1}\mu(r)\right)\right) \int_{|t|\le 2^k\mu(r)}f(y-t)\,dt\\
&= 2 \sum_{k=1}^Q 2^{k-1}\mu(r)\max\left(\varphi^*_r\left(2^{k-1}\mu(r)\right), \varphi^*_r\left(-2^{k-1}\mu(r)\right)\right) m_f(y,2^k\mu(r))\\
&\le 2\sum_{k=1}^Q \FI(r)m_f(y,2^k\mu(r)).
\emd
Then we split the domain of supremum in the followoing way:
\md9\label{1.sup-bound-1}
\sup_{\substack{|x-y|<\lambda(r)\\0<r<1}} \FI(r)m_f(y,A\mu(r))
&\le \sup_{\substack{|x-y|\le A\mu(r)\le\lambda(r)\\0<r<1}} \FI(r)m_f(y,A\mu(r))\\
&+   \sup_{\substack{A\mu(r)<|x-y|<\lambda(r)\\0<r<1}} \FI(r)m_f(y,A\mu(r)).
\emd
Notice that the second supremum is $T_A f(x)$. To estimate the first supremum, note that $|x-y|\le A\mu(r)\le\lambda(r)\le C\mu(r)\FI(r)^{-p}$ implies
\md1\label{1.varphir-bound-1}
\FI(r)\le  C^{1/p}A^{-1/p},
\emd
where $C$ is the constant from the condition $\mathit{2}$ of \lem{1.lemma-TAbound}. On the other hand, from $|x-y|\le A\mu(r)$ it follows $m_f(y,A\mu(r))\le Mf(x)$, which together with \e{1.varphir-bound-1}, \e{1.sup-bound-1} and \lem{1.lemma-TAbound} gives
\md9\label{1.term-bound}
\sup_{\substack{|x-y|<\lambda(r)\\0<r<1}} \FI(r)m_f(y,A\mu(r))
&\le C^{1/p}A^{-1/p} Mf(x) + T_A f(x)\\
&\le C^{1/p}A^{-1/p} (Mf^p(x))^{1/p} + \left(C\cdot\frac{Mf^p(x)}{A}\right)^{1/p}\\
&\le 2 C^{1/p}A^{-1/p} (Mf^p(x))^{1/p}.
\emd
Using \e{1.varphir-bound} and \e{1.term-bound} we get
\md2
&\sup_{\substack{|x-y|<\lambda(r)\\0<r<1}}\left|\int_{\mu(r)\le|t|\le\lambda(r)} \varphi_r(t)f(y-t)\,dt \right|\\
&\mspace{160mu}\le 2\sum_{k=1}^Q \sup_{\substack{|x-y|<\lambda(r)\\0<r<1}} \FI(r)m_f(y,2^k\mu(r))\\
&\mspace{160mu}\le 4C^{1/p}\sum_{k=1}^\infty 2^{-k/p} (Mf^p(x))^{1/p}\\
&\mspace{160mu}= \frac{4C^{1/p}}{2^{1/p}-1} (Mf^p(x))^{1/p},
\emd
which gives \e{1.sup-mulambda}.
\end{proof}

\begin{lemma}\label{1.rough-bound}
Let $\{\varphi_r\}$ be an arbitrary approximate identity and for some $1\le p<\infty$ the function $\lambda(r)$ satisfies
\md1\label{1.new-PIp-bound}
\sup_{0<r<1}\lambda(r)\|\varphi_r\|_q^p < \infty,
\emd
where $q=p/(p-1)$ is the conjugate index of $p$. Then for any function $f\in L^p(\ZT)$
\md1\label{1.sup-rough-bound}
\sup_{\substack{|x-y|<\lambda(r)\\0<r<1}}\left|\int_{|t|\le\lambda(r)} \varphi_r(t)f(y-t)\,dt \right|
\le C \left(M|f|^p(x)\right)^{1/p},\quad x\in\ZT,
\emd
where $C$ does not depend on function $f$.
\end{lemma}
\begin{proof}
The proof immediately follows from applying H\"{o}lder\a s inequality to the integral:
\md8
\sup_{\substack{|x-y|<\lambda(r)\\0<r<1}}\left|\int_{|t|\le\lambda(r)} \varphi_r(t)f(y-t)\,dt \right|
&\le \sup_{\substack{|x-y|<\lambda(r)\\0<r<1}}\|\varphi_r\|_q \cdot \left(\int_{|t|\le\lambda(r)} |f(y-t)|^p\,dt\right)^{1/p}\\
&\le \sup_{0<r<1}\|\varphi_r\|_q \cdot \left(\int_{|t|\le 2\lambda(r)} |f(x-t)|^p\,dt\right)^{1/p}\\
&\le \sup_{0<r<1}\|\varphi_r\|_q (4\lambda(r))^{1/p} \cdot \left(M|f|^p(x)\right)^{1/p},
\emd
which implies \e{1.sup-rough-bound} taking into account \e{1.new-PIp-bound}.
\end{proof}

\begin{lemma}\label{1.phi-r-bound}
If $\{\varphi_r\}$ is an arbitrary approximate identity, then for some $r_0\in(0,1)$
\md0
\frac{c}{\log\|\varphi_r\|_\infty} \le \FI(r) \le C_\varphi, \quad r_0<r<1,
\emd
where $c$ is a positive absolute constant.
\end{lemma}
\begin{proof}
Let $0<r<1$. Using the definitions of $\varphi^*_r(x)$ and $\FI(r)$ we conclude
\md0
\varphi_r(t) \le \varphi^*_r(t) \le \frac{\FI(r)}{|t|}, \quad t\in\ZT\setminus\{0\}.
\emd
Therefore, for a fixed $\delta\in(0,\pi)$ we have
\md8
1+o(1)
= \int_\ZT \varphi_r(t)\,dt
&\le \int_{|t|<\delta} \varphi^*_r(t)\,dt + \int_{\delta\le|t|\le\pi} \varphi^*_r(t)\,dt\\
&\le \|\varphi^*_r\|_\infty\int_{|t|<\delta}\,dt + \FI(r)\int_{\delta\le|t|\le\pi} \frac{dt}{|t|}\\
&= 2\delta\|\varphi_r\|_\infty + 2\FI(r)\log\frac{\pi}{\delta},
\emd
which implies
\md0
\FI(r) \ge \left(\frac{1}{2}+o(1)-\delta\|\varphi_r\|_\infty\right) \left(\log\frac{\pi}{\delta}\right)^{-1}.
\emd
Now, if we take $\delta = \pi/\|\varphi_r\|^2_\infty$, we get
\md0
\FI(r) \ge \left(\frac{1}{2}+o(1)-\frac{\pi}{\|\varphi_r\|_\infty}\right) \frac{1}{2\log\|\varphi_r\|_\infty},
\emd
which completes the proof of the first inequality (for example with $c=1/5$), since $\|\varphi_r\|_\infty\to\infty$ as $r\to1$. The second inequality can be deduced from the following:
\md0
\FI(r) = \sup_{x\in\ZT}|x\varphi^*_r(x)| \le \sup_{x\in\ZT}\left|\int_{|t|\le|x|}\varphi^*_r(t)\,dt\right| \le C_\varphi.
\emd
\end{proof}

\section{Proof of Theorems}\label{fatou-Lp}
\begin{proof}[Proof of \trm{1.weakpp}]
Without loss of generality we may assume that $f$ is non-negative.
Furthermore, we may assume that $\TPI_p\ge C_\varphi^p$ and $\lambda(r)\|\varphi_r\|_\infty\varphi_*^{p-1}(r)\ge C_\varphi^p$ for all $r\in(0,1)$. Otherwise, instead of $\lambda(r)$ we would define a new $\bar{\lambda}(r)$ as
\begin{equation*}
\bar{\lambda}(r) \deff \frac{\max\left(\TPI_p, C_\varphi^p\right)}{\|\varphi_r\|_\infty\varphi_*^{p-1}(r)} \ge \lambda(r), \quad 0<r<1,
\end{equation*}
for which those assumptions would hold. Denote $\mu(r)=\varphi_*(r)/\|\varphi_r\|_\infty$ and notice that
\begin{equation*}
\lambda(r) \ge \frac{C_\varphi^p}{\|\varphi_r\|_\infty\varphi_*^{p-1}(r)} \ge \frac{\varphi_*(r)}{\|\varphi_r\|_\infty}=\mu(r), \quad 0<r<1.
\end{equation*}

Let $x,y\in\ZT,\,0<r<1$ and $|x-y|<\lambda(r)$. We split the integral $\Phi_r(y,f)$ as follows
\md9\label{1.I123}
\Phi_r(y,f)
&= \int_\ZT \varphi_r(t)f(y-t)\,dt\\
&= \int_{|t|\le\mu(r)} \varphi_r(t)f(y-t)\,dt\\
&+ \int_{\mu(r)<|t|<\lambda(r)} \varphi_r(t)f(y-t)\,dt\\
&+ \int_{\lambda(r)\le|t|\le\pi} \varphi_r(t)f(y-t)\,dt = I^1 + I^2 + I^3.
\emd
First of all, from \lem{1.Phi*ineq1} we have
\md1\label{1.I3ineq}
\sup_{\substack{|x-y|<\lambda(r)\\0<r<1}} |I^3| \le 8 C_\varphi\cdot Mf(x).
\emd
Notice that from the condition $\TPI_p(\lambda, \varphi)<\infty$ it follows that
\md0
\lambda(r)\le \TPI_p\cdot\mu(r)\FI^{-p}(r).
\emd
Hence, from \lem{1.lemma-supmulambda} we get
\md1\label{1.I2ineq}
\sup_{\substack{|x-y|<\lambda(r)\\0<r<1}} |I^2| \le \frac{4\TPI_p^{1/p}}{2^{1/p}-1} \left(Mf^p(x)\right)^{1/p}.
\emd
Furthermore, using the definition of $\mu(r)$, for $I_1$ we obtain
\md8
|I^1|
&\le \int_{|t|\le\mu(r)} \varphi^*_r(t)f(y-t)\\
&\le \|\varphi_r\|_\infty \int_{-\mu(r)}^{\mu(r)}f(y-t)\,dt\\
&=  2\mu(r)\|\varphi_r\|_\infty m_f(y,\mu(r)) =  2\FI(r)m_f(y,\mu(r)),
\emd
where
\md0
m_f(y,t) = \frac{1}{2t}\int_{y-t}^{y+t}|f(u)|\,du, \quad y\in\ZT,\,t>0.
\emd
To estimate $I_1$ we split the supremum into two parts as we did in \lem{1.lemma-supmulambda}:
\md9\label{1.I1ineq}
\sup_{\substack{|x-y|\le\lambda(r)\\0<r<1}} |I^1|
&\le \sup_{\substack{|x-y|\le\mu(r)\le\lambda(r)\\0<r<1}} 2\FI(r)m_f(y,\mu(r))\\
&+ \sup_{\substack{\mu(r)<|x-y|<\lambda(r)\\0<r<1}} 2\FI(r)m_f(y,\mu(r))
\emd
Notice that the second supremum is $T_1 f(x)$, which can be estimated due to \lem{1.lemma-TAbound}. To estimate the first one, note that $\mu(r)\le\lambda(r)\le \TPI_p\mu(r)\FI(r)^{-p}$ implies
\md1\label{1.varphi-ineq}
\FI(r)\le\TPI_p^{1/p}.
\emd
On the other hand, $|x-y|\le\mu(r)$ implies $m_f(y,\mu(r))\le Mf(x)$, which together with \e{1.I1ineq}, \e{1.varphi-ineq} and \lem{1.lemma-TAbound} gives
\md9\label{1.I1ineq-1}
\sup_{\substack{|x-y|\le\mu(r)\le\lambda(r)\\0<r<1}} |I^1| \le 2\TPI_p^{1/p} \left(Mf^p(x)\right)^{1/p}.
\emd
Then, combining \e{1.I3ineq}, \e{1.I2ineq}, \e{1.I1ineq-1} and \e{1.I123}, we get
\md0
\sup_{\substack{|x-y|<\lambda(r)\\0<r<1}} \Phi_r(y,f)
\le \left( 2\TPI_p^{1/p} + \frac{4\TPI_p^{1/p}}{2^{1/p}-1} + 8 C_\varphi\right) \left(Mf^p(x)\right)^{1/p},
\emd
which implies \e{1.Phi*-bound}.
\end{proof}

\begin{proof}[Proof of \trm{1.fatouLp-example}]
From \e{1.extraCondition-phi} it follows that there exists $r_0\in(0,1)$ such that
\md1\label{1.cvarphi}
\frac{1}{\FI(r)}\int_{-\mu(r)}^{\mu(r)}|\varphi_r(t)|\,dt \ge \frac{c_\varphi}{2}
\emd
for any $r,\,r_0<r<1$. Denote
\md5
n(r) = \left[\frac{4\pi}{\lambda(r)}\right]\in\ZN,\label{1.nr}\\
\Delta_r = \bigcup_{k=0}^{n(r)-1}\left[\frac{2\pi k}{n(r)} - \mu(r), \frac{2\pi k}{n(r)} + \mu(r)\right],\label{1.Deltar}\\
\Lambda(r) = \lambda(r)\|\varphi_r\|_\infty\FI^{p-1}(r),\label{1.Lambdar}
\emd
If $x\in\ZT$ is an arbitrary point and $r_0<r<1$, then
\md0
x\in\left[\frac{2\pi k_0}{n(r)}, \frac{2\pi(k_0+1)}{n(r)}\right)
\emd
for some $k_0\in\{0, 1, \dots, n(r)-1\}$. Consider the function
\md1\label{1.fr-def}
f_r(x) = \frac{\Lambda^{1/p}(r)}{\FI(r)}\ZI_{\Delta_r}(x) \cdot \sgn\varphi_r\left(\frac{2\pi k_0}{n(r)}-x\right).
\emd
Note that
\md9\label{1.fr-Lp}
\|f_r\|_p^p
= \frac{\Lambda(r)}{\FI^p(r)}\cdot |\Delta_r| &\le \frac{\Lambda(r)}{\FI^p(r)}\cdot 2\mu(r)n(r)\\
&\le \frac{\lambda(r)\|\varphi_r\|_\infty\FI^{p-1}(r)}{\FI^p(r)}\cdot\frac{2\FI(r)}{\|\varphi_r\|_\infty}\cdot\frac{4\pi}{\lambda(r)} = 8\pi.
\emd
Clearly, taking $\theta = x-2\pi k_0/n(r)$, from \e{1.nr} we obtain
\md1\label{1.theta-bound}
|\theta| < \frac{2\pi}{n(r)} < \lambda(r).
\emd
Using the condition $\PI_p(\lambda, \varphi)=\infty$ and \lem{1.BV-means}, we may fix a sequence $r_k\nearrow 1$ such that
\md3
&\Lambda(r_k) > \left[2^{k+1}C_\varphi\cdot c_\varphi^{-1}\left(8\pi + k +\max_{1\le j<k}\frac{\Lambda^{1/p}(r_j)}{\FI(r_j)}\right)\right]^p,\quad  k=1,2,\ldots ,\label{1.Lambda-bound}\\
&\sup_{\theta\in\ZT}\frac{1}{|\Delta_{r_j}|}\int_{\Delta_{r_j}} \varphi^*_{r_k}(\theta-t)\,dt \le C_\varphi,\quad k=1, 2, \dots, j-1.\label{1.phi-bound-uniform}
\emd
In order to use \lem{1.BV-means} and get bounds (\ref{1.phi-bound-uniform}) we need to assure the assumption $\mu(r_k)<\frac{\pi}{n_k}$ holds. Notice that from (\ref{1.Lambda-bound}) and \lem{1.phi-r-bound} we have $\Lambda(r_k)>4C_\varphi^p\ge 4\varphi_*^p(r_k)$ which implies $\mu(r_k)<\frac{\lambda(r_k)}{4}\le \frac{\pi}{n_k}$.
Using \e{1.cvarphi} and \e{1.fr-def}, we get
\md9\label{1.Phir-bound-below}
\Phi_r(x-\theta,f_r)
&= \int_{\ZT} \varphi_r\left(\frac{2\pi k_0}{n(r)}-t\right)f_r(t)\,dt\\
&= \frac{\Lambda^{1/p}(r)}{\FI(r)} \int_{\Delta_r} \left|\varphi_r\left(\frac{2\pi k_0}{n(r)}-t\right)\right|\,dt\\
&\ge \frac{\Lambda^{1/p}(r)}{\FI(r)} \int_{2\pi k_0/n(r)-\mu(r)}^{2\pi k_0/n(r)+\mu(r)} \left|\varphi_r\left(\frac{2\pi k_0}{n(r)}-t\right)\right|\,dt\\
&= \frac{\Lambda^{1/p}(r)}{\FI(r)} \int_{-\mu(r)}^{\mu(r)} \left|\varphi_r(u)\right|du
\ge \frac{c_\varphi}{2} \Lambda^{1/p}(r).
\emd
Define
\md0
f(x)=\sum_{k=1}^\infty 2^{-k}f_{r_k}(x)\in L^p(\ZT).
\emd
We split $\Phi_{r_k}(\theta,f)$ in the following way
\md9\label{1.Phi-split-Lp}
\Phi_{r_k}(\theta,f)
&= \sum_{j=1}^\infty 2^{-j}\Phi_{r_k}(\theta, f_{r_j})\\
&= \sum_{j=1}^{k-1} 2^{-j}\Phi_{r_k}(\theta, f_{r_j}) + 2^{-k}\Phi_{r_k}(\theta, f_{r_k}) + \sum_{j=k+1}^\infty 2^{-j}\Phi_{r_k}(\theta, f_{r_j})\\
&= S^1 + S^2 + S^3.
\emd
From \e{1.theta-bound}, \e{1.Lambda-bound}, and \e{1.Phir-bound-below} it follows that
\md9\label{1.bound-S2-Lp}
\sup_{y\in\lambda(r_k,x)} S^2
&= \sup_{y\in\lambda(r_k,x)} 2^{-k}\Phi_{r_k}(y,f_{r_k})\\
&\ge 2^{-k}\Phi_{r_k}(x-\theta,f_{r_k})\\
&\ge 2^{-k-1}c_\varphi\Lambda^{1/p}(r_k)
\ge C_\varphi \left(8\pi + k +\max_{1\le j<k}\frac{\Lambda^{1/p}(r_j)}{\FI(r_j)}\right).
\emd
Furthermore, using \e{1.fr-def} and property $\mathit{\Phi3}$, we get
\md9\label{1.bound-S1-Lp}
\sup_{\theta\in\lambda(r_k,x)} |S^1|
&= \sup_{\theta\in\lambda(r_k,x)} \left|\sum_{j=1}^{k-1} 2^{-j}\int_\ZT \varphi_{r_k}(\theta-t)f_{r_j}(t)\,dt\right|\\
&\le \sup_{\theta\in\lambda(r_k,x)} \sum_{j=1}^{k-1} 2^{-j}\frac{\Lambda^{1/p}(r_j)}{\FI(r_j)} \int_{\Delta_{r_j}}|\varphi_{r_k}(\theta-t)|\,dt\\
&\le \sum_{j=1}^{k-1} 2^{-j}\frac{\Lambda^{1/p}(r_j)}{\FI(r_j)} \int_\ZT\varphi^*_{r_k}(u)\,du
\le C_\varphi\cdot\max_{1\le j<k} \frac{\Lambda^{1/p}(r_j)}{\FI(r_j)}.
\emd
Finally, using \e{1.fr-Lp} and \e{1.phi-bound-uniform} we get
\md9\label{1.bound-S3-Lp}
\sup_{\theta\in\lambda(r_k,x)} |S^3|
&= \sup_{\theta\in\lambda(r_k,x)} \left|\sum_{j=k+1}^\infty 2^{-j}\int_\ZT \varphi_{r_k}(\theta-t)f_{r_j}(t)\,dt\right|\\
&\le \sup_{\theta\in\lambda(r_k,x)} \sum_{j=k+1}^\infty 2^{-j}\frac{\Lambda^{1/p}(r_j)}{\FI(r_j)} \int_{\Delta_{r_j}}|\varphi_{r_k}(\theta-t)|\,dt\\
&\le \sum_{j=k+1}^\infty 2^{-j}\frac{\Lambda^{1/p}(r_j)}{\FI(r_j)}|\Delta_{r_j}|\cdot\sup_{\theta\in\ZT}\frac{1}{|\Delta_{r_j}|} \int_{\Delta_{r_j}}\varphi^*_{r_k}(\theta-t)\,dt \le 8\pi\cdot C_\varphi.
\emd
So, from \e{1.bound-S1-Lp}, \e{1.bound-S2-Lp}, \e{1.bound-S3-Lp} and \e{1.Phi-split-Lp} it follows
\md0
\sup_{\theta\in\lambda(r_k,x)} \Phi_{r_k}(\theta,f)
\ge \sup_{\theta\in\lambda(r_k,x)} S^2 - \sup_{\theta\in\lambda(r_k,x)} |S^1| - \sup_{\theta\in\lambda(r_k,x)} |S^3|
\ge C_\varphi\cdot k,
\emd
which imples \e{1.Lp-divergence}.
\end{proof}

\section{Final Remarks}\label{final_remarks}
Observe that the bound $\PI_p(\lambda,\varphi)<\infty$ for $1<p<\infty$ determines the exact rate of $\lambda(r)$ only for approximate identities satisfying $c_\varphi>0$. In fact, from \lem{1.Phi*ineq1} and \lem{1.rough-bound} it follows that \trm{1.fatouLp} and \trm{1.weakpp} hold if we replace the condition $\PI_p(\lambda,\varphi)<\infty$ by
\md1\label{0.sup-rough-bound}
\limsup_{r\to1}\lambda(r)\|\varphi_r\|_q^p < \infty,
\emd
where $q=p/(p-1)$ is the conjugate index of $p$. Thus, \otrm{OldCarlsson} is valid for any approximate identity, not necessarily non-negative and without the regular level sets assumption. One can check that in case of $c_\varphi=0$, the bound \e{0.sup-rough-bound} can give better convergence regions than $\PI_p(\lambda,\varphi)<\infty$ does. However, in this case it is unclear what is the exact bound for $\lambda(r)$ ensuring almost everywhere $\lambda(r)-$convergence.

 \end{document}